\documentclass{article}

\usepackage[margin=1in]{geometry}
\usepackage{amsmath}
\usepackage{amssymb}
\usepackage{amsthm}
\usepackage{mathrsfs}
\usepackage{mathtools}
\usepackage{tikz}
\usepackage{cleveref}
\usepackage{tikz-cd}
\usepackage{mathtools}
\usepackage{verbatim}

\theoremstyle{theorem}
\newtheorem{thm}{Theorem}[section]
\crefname{thm}{Theorem}{Theorems}

\crefname{prop}{Proposition}{Propositions}

\crefname{lem}{Lemma}{Lemmas}

\theoremstyle{definition}

\crefname{defn}{Definition}{Definitions}

\newtheorem*{ack*}{Acknowledgements}

\newcommand{\mb}{\mathbb}

\newcommand{\wt}{\widetilde}

\title{On the slope of the moduli space of genus 15 and 16 curves}
\author{Dennis Tseng}

\begin{document}
\maketitle
\begin{abstract}
We revisit the work of Chang and Ran on bounding the slopes of $\overline{\mathscr{M}_{15}}$ and $\overline{\mathscr{M}_{16}}$, correct one of the formulas used at the conclusion of the argument, and recompute the lower bounds on the slopes, yielding $s(\overline{\mathscr{M}_{15}})>6.5$ but not for $\overline{\mathscr{M}_{16}}$. Our contribution only involves plugging in formulas.
\end{abstract}

\section{Introduction}

The slope of the moduli space of curves is an important invariant, giving consequence for the birational geometry of $\overline{\mathscr{M}_{g}}$ \cite{CFM13}. In particular, Chang and Ran used 1-parameter families of space curves constructed using monads to show the slopes of $\overline{\mathscr{M}_{15}}$ and $\overline{\mathscr{M}_{16}}$ exceed 6.5 \cite{ChangRan15,ChangRan16}. The main result of \cite{BDPP13}, together with the slope bounds of Chang and Ran, would imply $\overline{\mathscr{M}_{15}}$ and $\overline{\mathscr{M}_{16}}$ are uniruled (see also \cite[Theorem 2.7]{F09}). 

Our goal is to correct the computation at the conclusion of the argument in \cite[Section 3]{ChangRan15} of the slope of the family of space curves $\mathscr{Y}\subset \mb{P}^1\times\mb{P}^3$ given as the degeneracy locus of a a vector bundle. We find $s(\overline{\mathscr{M}_{15}})> 6.53$ instead of $6.66$ as originally claimed. Therefore, the qualitative result that $\overline{\mathscr{M}_{15}}$ is uniruled remains unchanged. In fact, it has since been shown that $\overline{\mathscr{M}_{15}}$ is rationally connected \cite{BV05}. However, the recomputed lower bound for $s(\overline{\mathscr{M}_{16}})$ using \cite{ChangRan16} is only about 6 instead of 6.567 as originally claimed, so the question of the uniruledness of $\overline{\mathscr{M}_{16}}$ is still open. 

\subsection{Acknowledgements}
The author would like to thank Ziv Ran for helpful comments and encouragement.

\section{Computation}
We begin with a correction of the formula in \cite[page 219]{ChangRan15}.\footnote{The main difference between the formula in \Cref{HT} and the original is that each instance of $c_1c_2$ and $c_1^2c_2$ is replaced by $c_3$ and $c_4$ respectively. Note, however, the sign of $c_1(M)$ in $(-c_1(M)+2c_1)c_3$ is also flipped in the corrected version.} It is a special case of the chern numbers of degeneracy locus computed in \cite{HarrisTu}.\footnote{There are two relevant sign errors in \cite{HarrisTu}. First, \cite[1.4]{HarrisTu} is valid if you replace $x_1,\ldots,x_{m-r}$ with the dual chern roots, as the proof in Section 2 immediately defines the $x_i$ to be the dual chern roots (this typo is also mentioned in \cite[page 833]{Farkas}). Also, the sign in front of $c_1(M)$ in $(-c_1(M)+2c_1)c_3$ is flipped in \cite[page 474]{HarrisTu}, which I suspect is why the sign is also flipped in \cite[page 219]{ChangRan15}.}
\begin{thm}[corrected form of {{\cite[page 219]{ChangRan15}}}]
\label{HT}
Let $M$ be a smooth variety of dimension 4 and $f:A\to B$ be a homomorphism between vector bundles of rank $a$ and $a+1$, respectively. Suppose the locus $Z\subset M$, where $f$ has rank $<a$, is a locally complete intersection surface. Then, the virtual Chern numbers of $Z$ are given by
\begin{align*}
    c_1(Z)^2 =& (c_1(M)-c_1)^2c_2-2(c_1(M)-c_1)c_3+c_4\\
    c_2(Z) =& (c_2(M)-c_1(M)c_1+c_2(A)-c_2(B)+c_1(B)^2-c_1(A)c_1(B))c_2+\\
    &+(-c_1(M)+2c_1)c_3+c_4. 
\end{align*}
where $c_i:=c_i(B-A)$. 
\end{thm}

\begin{thm}
\label{CR15}
The slope of $\overline{\mathscr{M}}_{15}$ is at least $\frac{98}{15}\approx 6.53$, so in particular  $\overline{\mathscr{M}}_{15}$ has Kodaira dimension $-\infty$.
\end{thm}

\begin{proof}
Applying \Cref{HT} to the case $M=\mathbb{P}^1\times\mathbb{P}^3$, $c(A)=c(\mathscr{O}^4)$ and $B=E(2)$, where $E$ is given as
\begin{align*}
    0\to E\to \mathscr{O}(1,0)^8\oplus \mathscr{O}(0,-1)\to \mathscr{O}(1,1)^4\to 0
\end{align*}
as in \cite[Example 1.6]{ChangRan15}, we find
\begin{align*}
    &c_1(Z)^2= 216 \quad c_2(Z)= 336 \quad\\
    &\kappa  = 328 \quad \delta = 392  \quad \lambda = 60,
\end{align*}
giving the claimed lower bound to the slope of $\overline{\mathscr{M}}_{15}$.
\end{proof}
However, this is not sufficient for the application to $\overline{\mathscr{M}_{16}}$ given in \cite{ChangRan16}. Instead, one gets
\begin{thm}
The slope of $\overline{\mathscr{M}_{16}}$ is at least $\frac{1472}{245}\approx 6.008$
\end{thm}

\begin{proof}
We will refer the reader to \cite{ChangRan16} for the details of the proof. We will just check one computation here. This is just Type $\beta$ family in \cite[page 271]{ChangRan16}, but there are typos in the formulas. Specifically, the second and fourth line of \cite[(1.3)]{ChangRan16} should read
\begin{align*}
    \beta(F,A_1,A_2)\cdot_{\overline{M_{i+1}}}\delta_j &= m_1m_2 F\cdot \delta_j\qquad \text{for }j\neq 0,1,i\\
    \beta(F,A_1,A_2)\cdot_{\overline{M_{i+1}}}\delta_0 &= m_1m_2F\cdot \delta_0+\sum_{\ell=1}^{2}\Big(m_{2-\ell}(m_\ell(2h-2)-(2g(A_\ell)-2)+A_\ell\cdot A_\ell)-A_1\cdot A_2\Big).
\end{align*}
In spite of this, our recomputed correction term $\beta(F,A_1,A_2)\cdot_{\overline{M_{i+1}}}\delta-m_1m_2F\cdot \delta$ specialized to our case agrees with the correction term $- 2(14 \cdot 220 + 16) + 16$ found in the formula for $F_{0,16}\cdot \delta$ on \cite[page 273]{ChangRan16}.

From the proof of \Cref{CR15}, Chang and Ran construct a surface $\mathscr{Y}\subset \mb{P}^1\times \mb{P}^3$, viewed as a family of curves over $\mb{P}^1$. Each member of $\mathscr{Y}\to \mb{P}^1$ is a degree 14 space curve of genus 15, and the image of $\mathscr{Y}\to \mb{P}^3$ is a degree 16 surface \cite[page 273]{ChangRan16}.

By pulling back generic hyperplanes in $\mb{P}^3$, we get two smooth multisections $A_1$ and $A_2$ of $\mathscr{Y}\to \mb{P}^1$ of degree 14 meeting transversely with $A_1^2=A_2^2=A_1\cdot A_2=16$. We can also assume $A_1$ and $A_2$ do not meet at points where either multisection is tangent to the fiber. By base changing under $B:=A_1\times_{\mb{P}^1}A_2\xrightarrow{\pi} \mb{P}^1$, we get a family $\pi^{*}\mathscr{Y}\to B$ with two sections $\sigma_1,\sigma_2$ mapping isomorphically onto $A_1,A_2\subset \mathscr{Y}$. Blowing up $\pi^{*}\mathscr{Y}$ at the (reduced) points of intersection of $\sigma_1$ with $\sigma_2$, we get nonintersecting sections $\wt{\sigma}_1,\wt{\sigma}_2$ of a family $\wt{\mathscr{Y}}\to B$. 

Now, we want to determine the slope of the map $\phi_B: B\to \overline{\mathscr{M}_{16}}$ given by $\wt{\mathscr{Y}}$ in terms of the map $\phi_{\mb{P}^1}: \mb{P}^1\to \overline{\mathscr{M}_{15}}$ given by $\mathscr{Y}$. We see 
\begin{align*}
    \phi_B^{*}\lambda = (14)^2\phi_{\mb{P}^1}^{*}\lambda
    \quad \phi_B^{*}\delta_1=16
    \quad \phi_B^{*}\delta_i=(14)^2\phi_{\mb{P}^1}^{*}\delta_i=0\quad\text{for $i>1$},
\end{align*}
so the only intersection left is $\phi_B^{*}\delta_0$. This differs from $(14)^2\phi_{\mb{P}^1}^{*}\delta_0$ by the sum of the chern numbers of the normal bundles of $\wt{\sigma}_1$ and $\wt{\sigma}_2$ \cite[page 147]{HM98}. To do this, we see that 
$$(\wt{\sigma}_i)^2=-\wt{\sigma}_i^{*}\omega_{\wt{\mathscr{Y}}/B}=-\sigma_i^{*}\omega_{\pi^{*}\mathscr{Y}/B}-A_1\cdot A_2=14 (-A_i\cdot\omega_{\mathscr{Y}/\mb{P}^1})-16, $$
where $A_i\cdot\omega_{\mathscr{Y}/\mb{P}^1}$ can be computed using adjunction on $\mathscr{Y}$ to be
\begin{align*}
    A_i\cdot\omega_{\mathscr{Y}}-(14)(c_1(\omega_{\mb{P}^1}))=(2g(A_i)-2)-A_i^2-(14)(-2).
\end{align*}
Therefore, $\phi_B^{*}\lambda=60\cdot 14^2$ and $(\wt{\sigma}_i)^2=14 (-(15\cdot 14 - 2) + (2\cdot 0 - 2)\cdot 14 + 16) - 16=-3096$, and
\begin{align*}
   \frac{\phi_B^{*}\delta}{\phi_B^{*}\lambda}=\frac{\phi_B^{*}\delta_0+\phi_B^{*}\delta_1}{\phi_B^{*}\lambda}=\frac{(14^2\phi^{*}_{\mb{P}^1}\delta_0+(\wt{\sigma}_1)^2+(\wt{\sigma}_2)^2)+16}{60\cdot 14^2}=\frac{1472}{245}.
\end{align*}
\end{proof}

\bibliographystyle{alpha}
\bibliography{references.bib}

\begin{thebibliography}{BDPP13}

\bibitem[BDPP13]{BDPP13}
S\'{e}bastien Boucksom, Jean-Pierre Demailly, Mihai P\u{a}un, and Thomas
  Peternell.
\newblock The pseudo-effective cone of a compact {K}\"{a}hler manifold and
  varieties of negative {K}odaira dimension.
\newblock {\em J. Algebraic Geom.}, 22(2):201--248, 2013.

\bibitem[BV05]{BV05}
Andrea Bruno and Alessandro Verra.
\newblock {$\mathscr{M}_{15}$} is rationally connected.
\newblock In {\em Projective varieties with unexpected properties}, pages
  51--65. Walter de Gruyter, Berlin, 2005.

\bibitem[CFM13]{CFM13}
Dawei Chen, Gavril Farkas, and Ian Morrison.
\newblock Effective divisors on moduli spaces of curves and abelian varieties.
\newblock In {\em A celebration of algebraic geometry}, volume~18 of {\em Clay
  Math. Proc.}, pages 131--169. Amer. Math. Soc., Providence, RI, 2013.

\bibitem[CR86]{ChangRan15}
Mei-Chu Chang and Ziv Ran.
\newblock The {K}odaira dimension of the moduli space of curves of genus
  {$15$}.
\newblock {\em J. Differential Geom.}, 24(2):205--220, 1986.

\bibitem[CR91]{ChangRan16}
Mei-Chu Chang and Ziv Ran.
\newblock On the slope and {K}odaira dimension of {$\overline M_g$} for small
  {$g$}.
\newblock {\em J. Differential Geom.}, 34(1):267--274, 1991.

\bibitem[Far09a]{F09}
Gavril Farkas.
\newblock Birational aspects of the geometry of {$\overline{\mathscr M}_g$}.
\newblock In {\em Surveys in differential geometry. {V}ol. {XIV}. {G}eometry of
  {R}iemann surfaces and their moduli spaces}, volume~14 of {\em Surv. Differ.
  Geom.}, pages 57--110. Int. Press, Somerville, MA, 2009.

\bibitem[Far09b]{Farkas}
Gavril Farkas.
\newblock Koszul divisors on moduli spaces of curves.
\newblock {\em Amer. J. Math.}, 131(3):819--867, 2009.

\bibitem[HM98]{HM98}
Joe Harris and Ian Morrison.
\newblock {\em Moduli of curves}, volume 187 of {\em Graduate Texts in
  Mathematics}.
\newblock Springer-Verlag, New York, 1998.

\bibitem[HT84]{HarrisTu}
J.~Harris and L.~Tu.
\newblock Chern numbers of kernel and cokernel bundles.
\newblock {\em Invent. Math.}, 75(3):467--475, 1984.

\end{thebibliography}

\end{document}